\documentclass[12pt,a4paper]{article}

\usepackage{amsfonts}     
\usepackage{amsmath}
\usepackage{amssymb}
\usepackage{amstext}
\usepackage{amsthm}  
\usepackage{graphicx}

\newcommand{\R}{\mathbb R}

\newcommand{\eps}{\varepsilon}

\newtheorem{teo}{Theorem}

\newtheorem{lemma}[teo]{Lemma}
\newtheorem{prop}[teo]{Proposition}

\theoremstyle{definition}

\theoremstyle{remark}
\newtheorem{rem}[teo]{Remark}

\newcommand{\D}{{\cal D}^{1,2}}
\newcommand{\Ne}{{\cal N}}
\newcommand{\V}{{\cal V}}

\usepackage{mathrsfs}
\begin{document}

\title{Solutions for a Nonhomogeneous Nonlinear Schroedinger Equation
with Double Power Nonlinearity}
\author{M.Ghimenti
\thanks{Dipartimento di Matematica Applicata, Universit\`a di Pisa, via Buonarroti 1c,
56127, Pisa, Italy},
A.M.Micheletti
\thanks{Dipartimento di Matematica Applicata, Universit\`a di Pisa, via Buonarroti 1c,
56127, Pisa, Italy}}
\date{}
\maketitle

\begin{abstract}
\noindent We consider the problem $- \Delta u+V(x)u=f^{\prime}(u)+g(x)$ in
$\R^N$, under the assumption $\lim_{x\rightarrow\infty}V(x)=0$, and with
the non linear term $f$ with a {\em double power} behavior. We prove the
existence two solutions when $g$ is sufficiently small and $V<0$.

\end{abstract}

\section{Perturbation of NSE}

We consider the existence of solutions of the following nonhomogeneous problem

\begin{equation}
\left\{
\begin{array}{ll}
\tag{$\mathscr P$}-\Delta u+V(x)u=f^{\prime }(u)+g(x), & x\in \R^{N}; \\
E^{V}_g(u)<\infty . &
\end{array}
\right.  \label{PV}
\end{equation}

where the energy functional is defined by
\begin{equation*}
E_g(u)=E^{V}_g(u)=\frac{1}{2}\int\limits_{\R^{N}}|\nabla u|^{2}+
V(x)u^{2}(x)dx-
\int\limits_{\R^{N}}f(u)dx-\int\limits_{\R^{N}}g(x)u(x)dx.
\end{equation*}

The nonlinearity is given by a function $f$ of {\em double power} type that is
an even function $f\in C^{3}(\R,\R)$ with
$f(0)=f'(0)=f''(0)=0$ satisfying the
following requirements:

\begin{enumerate}
\item there exist positive numbers $c_{0},c_{2},p,q$ with $2<p<2^{\ast }<q$
such that

\begin{equation}
\tag{$f_0$}
\left\{
\begin{array}{ll}
c_{0}|s|^{p}\leq f(s) & \text{ for }|s|\geq 1; \\
c_{0}|s|^{q}<f(s) & \text{ for }|s|\leq 1;
\end{array}
\right.  \label{f1}
\end{equation}
\begin{equation}
\tag{$f_2$}
\left\{
\begin{array}{ll}
|f^{\prime \prime }(s)|\leq c_{2}|s|^{p-2} & \text{ for }|s|\geq 1; \\
|f^{\prime \prime }(s)|\leq c_{2}|s|^{q-2} & \text{ for }|s|\leq 1;
\end{array}
\right.  \label{f2}
\end{equation}

\item there exists $\mu _{1}>2$ and $\mu _{2}>1$ such that, for all $s\neq 0$
\begin{equation}
\tag{$f_\mu$}
\begin{array}{ccccc}
0<\mu _{1}f(s)\leq f^{\prime }(s)s ,\label{fmu}
&&
\mu _{2}f^{\prime }(s)s<f^{\prime \prime }(s)s^{2},
 &&
f^{\prime \prime \prime }(s)s^{3}>0;
\end{array}
\end{equation}

\item for any $u\in\D$ we have
\begin{equation} \tag{$f_3$}\label{f3}
f'''(u)u^3\in L^1.
\end{equation}
\end{enumerate}
For example the required assumptions are satisfied by $f(s)=\frac{|s|^q}{1+|s|^{q-p}}$ with $q-p$
small enough, as shown in the appendix.

We assume $V\in L^{N/2}(\R^N)\cap L^t$, for some $t>N/2$ and
\begin{equation}\label{condV}
||V||_{L^{N/2}}<S:=\inf_{u\in{\D}}\frac{\displaystyle\int_{\R^N}|\nabla u|^2}
{\displaystyle\left(\int_{\R^N}|u|^{2^*}\right)^{2/2^*}}.
\end{equation}
Moreover, we want $V\leq0$ and $V<0$ on a set of positive measure.

In \cite{Zhu91} the existence of two positive solutions $u_1,u_2\in H^1(\R^N)$ of the equation
$-\Delta u+u=|u|^{p-2}u+g$ is proved when $g\in L^2$ satisfies
$0\leq g\leq C\exp(-(1+\varepsilon)|x|)$, $g\not\equiv 0$.

Recently, in \cite{Zho00}, a similar problem for the $p$-laplacian is studied. Namely,
the author proves, with variational techniques, that the problem
$-\Delta_p u+c|u|^{p-2}u=|u|^{p^*-2}u+f(x,u)+h(x)$ in $\R^N$,
where $2\leq p<N$, $c>0$,
$h\in W^{-1,p^\prime}(\R^N)$ and $f$ a is continuous superlinear function such that
$f(x,0)=0$ and $f(x,u)=o(|u|^{p^*-1})$ as $|u|\to\infty$,
admits two positive solutions $u_1,u_2\in H^1(\R^N)$.

The existence of a positive solution of the problem $-\Delta u+u=|u|^{p-1}u+g$ on $\R^N$, $u(x)\rightarrow 0$ for
$|x|\rightarrow \infty$, was proven in \cite{Be96} when $p>\frac{N}{N-2}$ and $g\in C^{0,\alpha}(\R^N)$, $g\geq0$,
$g\not\equiv0$  and $g(x)\leq \frac{C}{(1+|x|^2)^{p/p-1}}$ for some $C>0$. In \cite{BN01} there is a
result of multiplicity for this problem.

In \cite{Ta92} the author shows that the Dirichlet problem on a bounded domain $\Omega\subset \R^N$ in the
critical case
$
-\Delta u=|u|^{2^*-2}u+g
$
has two solutions $u_0,u_1\in H_0^1(\Omega)$, for $g$ satisfying a suitable condition, and if $g\geq 0$ then
$u_0\geq0$ and $u_1\geq0$.

We are interested in studying the problem with double power nonlinearity.

In pioneering work Berestycki and Lions \cite{BL83a,BL83b}
showed the existence of a
positive solution in the case $V\equiv0$ when $f''(0)=0$, $f$ has a
supercritical growth near the origin and subcritical at infinity.

More recently in the papers \cite{BR05,BGM04,BM04,Pis} the double-power
growth condition has been used to obtain the existence of positive solutions for
different problems of the tipe (\ref{PV}).
In particular, in \cite{BGM04}, the authors proved
that in the same hypothesis on $V$ the homogeneous problem
\begin{equation}\label{ph}
-\Delta u+Vu=f'(u)
\end{equation}
has a ground state solution (i.e. least energy nontrivial solution).
Other results on similar problems with the double power nonlinearity can be found in
\cite{AP,BR05,GM06}.

In this paper we prove the following theorem
\begin{teo}
If $g\in L^{\frac{2N}{N+2}}\cap L^{s}$, for some $s>\frac{2N}{N+2}$,
and if $||g||_{\frac{2N}{N+2}}$ is sufficiently small there exist two solutions
of problem (\ref{PV}) in $\D$. The first solution is close to 0; if
also $||g||_{L^{p'}\cap L^{q'}}$is small enough, the critical value of the second solution is
close to the least energy level $m_V$ of the homogeneous problem (\ref{ph}).

Furthermore, if $g\geq0$ the two solutions are non negative.
\end{teo}

\begin{rem}
Indeed the hypothesis on the sign of $V$ is used only to find the second solution, but we prefer a more compact
claim for the theorem. Anyway, in the proofs we focus out when we use any hypothesis.
\end{rem}
To get the solutions of (\ref{PV}) we look for critical points of the functional $E_g^V$ constrained on the
Nehari manifold
\begin{eqnarray*}
\Ne_g^V&=&\Ne_g=\left\{u\in\D\::\ \langle\nabla E_g(u),u\rangle=0, u\neq 0\right\}=\\
&=&\left\{u\in\D\smallsetminus0:\int_{\R^N}\!|\nabla u|^2+\int_{\R^N}\! Vu^2-
\int_{\R^N}\! f'(u)u-\int_{\R^N}\! gu=0\right\}.
\end{eqnarray*}
The study of the structure of the Nehari manifold will be a fundamental part of this paper.

This paper is organized as follows:
in section 2, we recall some technical results concerning the appropriate function space required
by the growth properties of the nonlinearity $f$ . Moreover, we study the geometry and the
properties of the Nehari manifold. In section 3,
we prove a Splitting Lemma necessary to overcome the lack of compactness. This lemma is a variant
for a well known result of \cite{St84}. In section 4 we prove the existence of two distinct critical
points of the functional on the
Nehari manifold.

\section{Notations and preliminary result}

We will use the following notations

\begin{itemize}
\item {$\D=\D(\R^N)=$ completion of $C_0^\infty(\R^N)$ with respect to the norm
$$||u||=\left(\int_{\R^N}|\nabla u|^2\right)^{1/2};$$}
\item $\displaystyle||u||^2_V=\int_{\R^N} |\nabla u|^2+ \int_{\R^N} Vu^2$;
notice that, by (\ref{condV}),
we have that $||u||_V$
is a norm in $\D$ equivalent to the usual one;
\item $2^*=\frac{2N}{N-2}$;
\item $m_g=\inf\limits_{u\in \Ne_g^V}E_g^V(u)$;
\item $m_{1,g}=\inf\limits_{u\in \Ne_g^-}E_g^V(u)$;
\item $m_0=\inf\limits_{u\in \Ne_0^0}E_0^0(u)$; we call $\omega$
the minimizer of $E_0^0$
on $\Ne_0^0$ radially symmetric;
\item $m_V=\inf\limits_{u\in \Ne_0^V}E_0^V(u)$;
we call $\bar u$ the minimizer of $E_0^V$ on $\Ne_0^V$;
\item $\Gamma_u=\{x\in\R^N\ :\ |u(x)|>1\}$;
\item $|A|=$ the Lebesgue measure of the subset $A\subset \R^N$;
\item $B_R=\{x\in\R^N\ :\ |x|\leq R\}$;
\item $B_R^C=\R^N\smallsetminus B_R$;
\item $u_y(x)=u(x+y)$.
\end{itemize}

In order to study the properties of the functional $E_g^V$ and its Nehari manifold,
we consider some suitable Orlicz space $L^p+L^q$,
where $2<p<2^*<q$, related to the double power growth behavior of the function $f$.
We recall some properties of these spaces to get the smoothness of the functional
$E_g^V$

Given $p\neq q$, we consider the space $L^p+L^q$ made up of the
functions $v:\R^N\rightarrow\R$ such that
\begin{equation}
v=v_1+v_2\ \text{ with }\ v_1\in L^p, v_2\in L^q.
\end{equation}
The space $L^p+L^q$ is a Banach space equipped with the norm:
\begin{equation}
||v||_{L^p+L^q}=\inf\{\ ||v_1||_{L^p}+||v_2||_{L^q}\ :\ v_1\in L^p, v_2\in L^q,\
v_1+v_2=v\}.
\end{equation}
It is well known (see, for example \cite{BL76}) that $L^p+L^q$
coincides with the dual of $L^{p'}\cap L^{q'}$. Then:
\begin{equation}
L^p+L^q=\left(L^{p'}\cap L^{q'} \right)'\ \text{ with }\ p'=\frac p{p-1},\ q'=\frac q{q-1},
\end{equation}
and we can introduce the following norm equivalent to the previous one
\begin{equation}
||v||_{L^p+L^q}=\inf_{\varphi\neq0}\frac{\int v\varphi}{||\varphi||_{L^{p'}}+||\varphi||_{L^{q'}}}.
\end{equation}
Hereafter we recall some results useful for this paper contained in
\cite{BF04,BM04}.

\begin{lemma}\label{lp+lq}
We have
\begin{enumerate}
\item if $v\in L^p+L^q$, the following inequalities hold:
\begin{eqnarray*}
&&\max\left[||v||_{L^q(\R^N\smallsetminus\Gamma_v)}-1,
\frac{1}{1+|\Gamma_v|^\frac1\tau}||v||_{L^p(\Gamma_v)}\right]\leq\\
&\leq&||v||_{L^p+L^q}\leq\\
&\leq&\max[||v||_{L^q(\R^N\smallsetminus\Gamma_v)},||v||_{L^p(\Gamma_v)}]
\end{eqnarray*}
when $\tau=\frac{pq}{q-p}$;
\item let $\{v_n\}\subset L^p+L^q$. Then $\{v_n\}$ is bounded in $L^p+L^q$ if and only
if the sequences $\{|\Gamma_{v_n}|\}$ and
$\{||v||_{L^q(\R^N\smallsetminus\Gamma_{v_n})}+||v||_{L^p(\Gamma_{v_n})}\}$
are bounded.
\item $f'$ is a bounded map from $L^p+L^q$ into $L^\frac{p}{p-1}\cap L^\frac{q}{q-1}$
\end{enumerate}
\end{lemma}
\begin{rem}\label{emb}
By the previous lemma we have
$L^{2^*}\subset L^p+L^q$ when $2<p<2^*<q$. Then, by Sobolev inequality,
we get the continuous embedding
\begin{equation*}
\D(\R^N)\subset L^p+L^q.
\end{equation*}
\end{rem}

In order to prove the $C^2$ regularity of the functional $E_g^V$, we need the following lemmas proved in
\cite{BM04}
\begin{lemma}\label{stimef}
If $f$ satisfies the hypothesis (\ref{f1}) and (\ref{f2}), we have that
\begin{enumerate}
\item if $\theta,u$ are bounded in $L^p+L^q$, then $f''(\theta)u$ is bounded in
$L^{p'}\cap L^{q'}$;
\item $f''$ is a bounded map from $L^p+L^q$ into $L^{p/p-2}\cap L^{q/q-2}$;
\item $f''$ is a continuous  map from $L^p+L^q$ into $L^{p/p-2}\cap L^{q/q-2}$;
\item the map $(u,v)\mapsto uv$ from $(L^p+L^q)^2$ in $L^{p/2}+ L^{q/2}$ is bounded.
\end{enumerate}
\end{lemma}

\begin{lemma}The functional $E_g^V$ is of class $C^2$ and it holds
\begin{eqnarray}
E_g'(u)[v]=\langle \nabla E_g^V(u),v\rangle&=&\int_{\R^N}\nabla u\nabla v+Vuv-f'(u)v-gv;\\
E_g''(u)[v,w]&=&\int_{\R^N} \nabla v\nabla w +Vvw -f''(u)vw.
\end{eqnarray}
Moreover the Nehari manifold defined as
\begin{equation}
\Ne_g^V=\left\{u\in\D\smallsetminus 0\  :\ \int_{\R^N}|\nabla u|^2+Vu^2-f'(u)udx-gu=0\right\}
\end{equation}
is of class $C^1$ and its tangent space at the point $u$ is
\begin{equation*}
T_u{\Ne_g^V}=\left\{v\in\D:\int_{\R^N}2\nabla u\nabla v+2Vuv-f'(u)vdx-f''(u)uv-gv=0\right\}.
\end{equation*}
\end{lemma}

At last, we introduce the functions
\begin{eqnarray}
\varphi_0^u(t)&=&\varphi_0(t):=E_0^V(tu)=
\int_{\R^N}\frac{t^2}{2}(|\nabla u|^2+Vu^2)-f(tu);\label{phi0}\\
\varphi_g^u(t)&=&\varphi_g(t):=E_g^V(tu)=
\varphi_0(t)-t\int_{\R^N} gu.\label{phig}
\end{eqnarray}
We have that
\begin{eqnarray}
\varphi_g'(t)&=& t||u||_V^2- \int_{\R^N} f'(tu)u-\int_{\R^N} gu;\\
\varphi_g''(t)&=&||u||_V^2-\int_{\R^N} f''(tu)u^2;\\
\varphi_g'''(t)&=&-\int_{\R^N} f'''(tu)u^3.
\end{eqnarray}

Notice that the conditions on $f$ assure that also $\varphi_g'''(t)$ exists.
Furthermore, if $\frac{d}{dt}\varphi_g(\bar t)=0$, then
$\langle\nabla E(\bar tu),u\rangle=0$, so $\bar tu\in\Ne_g^V$, and {\em vice versa},
so we want to find the critical points of $\varphi_g(t)$.

To study the manifold $\Ne_g^V$ it is useful to consider the following manifold:
\begin{equation}
\V=\left\{ w\neq0:G(w):=||w||_V^2-\int_{\R^N} f''(w)w^2=0\right\}.
\end{equation}

\begin{lemma}
We have that for all $u\in\D$ there exists an unique $T_u>0$ such that
$T_uu\in\V$
\end{lemma}
\begin{proof}
We have that,
using (\ref{fmu}) and (\ref{f1}),
\begin{eqnarray}
\varphi_0'(t)&=&t||u||_V^2- \int_{\R^N} f'(tu)u\leq
t||u||_V^2- \frac{\mu_1}{t}\int_{\R^N} f(tu)\leq\nonumber\\
&\leq&t||u||_V^2-t^{q-1}c_0\mu_1\int_{t|u|<1}|u|^q
-t^{p-1}c_0\mu_1\int_{t|u|\geq 1}|u|^p\leq\label{phidiv}\\
&\leq&t||u||_V^2-t^{p-1}c_0\mu_1
\int_{|u|\geq 1}|u|^p\rightarrow-\infty\text{ when } t\rightarrow\infty,
\nonumber
\end{eqnarray}
because $p>2$. Furthermore we have that $\varphi_0'(t)$ is strictly concave when $t\neq 0$, and that
$\varphi_0''(0)>0$, so, for every $u\in\D$ there exist an
unique maximum point $T_u>0$ for the function $\varphi_0'(t)$.
Thus
\begin{equation*}
0=T_u^2\varphi_0''(T_u)=||T_uu||_V^2-\int_{\R^N} f''(T_uu)(T_uu)^2.
\end{equation*}
\end{proof}
\begin{prop}\label{infV}
We have that $\inf\limits_{w\in\V}||w||_V^2>0.$
\end{prop}

\begin{proof}

By contradiction, we suppose that there exists a sequence $\{w_n\}_n\subset \V$ such that
$||w_n||_V^2$ converges to 0. We set $t_n=||w_n||_V$, hence we can write $w_n=t_nv_n$
where $||v_n||_V=1$. Remark \ref{emb}, the sequence is bounded in
$L^p+L^q$. Since $w_n\in\V$ and $t_n$ converges to 0, we have
\begin{eqnarray*}
1&=&||v_n||_V^2=\frac{||w_n||_V^2}{t_n^2}=
\frac 1{t_n^2}\int\limits_{\R^N} f''(t_nv_n)v_n^2\leq\\
&&\leq c_2t_n^{q-2}\int\limits_{\R^N\smallsetminus \Gamma_{t_nv_n}}|v_n|^q+
c_2t_n^{p-2}\int\limits_{\Gamma_{t_nv_n}}|v_n|^p\leq\\
&&\leq c_2t_n^{q-2}\int\limits_{\R^N\smallsetminus \Gamma_{t_nv_n}}|v_n|^q+
c_2t_n^{p-2}\int\limits_{\Gamma_{v_n}}|v_n|^p\leq\\
&&\leq c_2t_n^{q-2}\!\!\!\!\int\limits_{\R^N\smallsetminus \Gamma_{v_n}}\!\!\!\!|v_n|^q+
c_2t_n^{q-2}\!\!\!\!\!\!\!\!\int\limits_{(\R^N\smallsetminus \Gamma_{t_nv_n})\cap\Gamma_{v_n}}
\!\!\!\!\!\!\!\!\frac{|v_n|^p}{t_n^{q-p}}
+c_2t_n^{p-2}\int\limits_{\Gamma_{v_n}}|v_n|^p\leq\\
&&\leq c_2t_n^{q-2}\int\limits_{\R^N\smallsetminus \Gamma_{v_n}}|v_n|^q+
2c_2t_n^{p-2}\int\limits_{\Gamma_{v_n}}|v_n|^p.
\end{eqnarray*}
Hence we get
\begin{displaymath}
1\leq c_2t_n^{q-2}\int\limits_{\R^N\smallsetminus \Gamma_{v_n}}|v_n|^q+
2c_2t_n^{p-2}\int\limits_{\Gamma_{v_n}}|v_n|^p
\end{displaymath}
and by claim 2 of Remark \ref{lp+lq} we get the contradiction.
\end{proof}
\begin{lemma}\label{L}
Let $u\in\D$ and let $T_u$ the unique positive number such that $T_uu\in\V$. Then
\begin{equation}
L=\inf_{||u||_V=1}T_u-\int_{\R^N} f'(T_uu)u>0.
\end{equation}
\end{lemma}
\begin{proof}
By contradiction, suppose that there exists a
minimizing sequence $u_n$,
with $||u_n||_V=1$ such that $\displaystyle T_{u_n}-\int f'(T_{u_n}{u_n}){u_n}:=\sigma_n\rightarrow0$.
Let $w_n=T_{u_n}u_n$. We have that
\begin{equation*}
T_{u_n}^2=||w_n||_V^2=\int f''(w_n)w_n^2,
\end{equation*}
because $w_n\in\V$. Furthermore, by hypothesis, we have
\begin{equation*}
||w_n||_V=\int f'(w_n)\frac{w_n}{||w_n||_V}+\sigma_n.
\end{equation*}
Thus, by \ref{fmu},
\begin{eqnarray*}
\mu_2||w_n||_V^2&=&\mu_2\int f'(w_n)w_n+\mu_2\sigma_n||w_n||_V<\\
&<&\int f''(w_n)w_n^2+\mu_2\sigma_n||w_n||_V=\\
&=&||w_n||_V^2+\mu_2\sigma_n||w_n||_V.
\end{eqnarray*}
So, because $\mu_2>1$ we have that
\begin{equation}
0<(\mu_2-1)||w_n||_V<\mu_2\sigma_n\rightarrow0,
\end{equation}
that is a contradiction.
\end{proof}
\begin{rem}Obviously, by Lemma \ref{infV}
we have also
$$
B:=\inf_{||u||_V=1}T_u>0,
$$
and $B$ does not depend on $g$.
\end{rem}
At last we can give the following characterization of the Nehari manifold.
\begin{prop}\label{gsmall}
Let $||g||_{L^{\frac{2N}{N+2}}}$, sufficiently small, and let $u\in\D$ with $||u||_V=1$.
Then
\begin{enumerate}
\item If $\displaystyle\int\! gu<0$, then there exists an unique $t^1_u$ such that $t^1_uu\in\Ne_g^V$ and
$t_u^0<t^1_u$, where $t_u^0$ is the unique value for which $t_u^0\in\Ne_0^V$.
\item If $\displaystyle\int\! gu=0$, then there exists an unique $t^1_u$ such that $t^1_uu\in\Ne_g^V$ and
$t_u^0=t^1_u$.
\item If $\displaystyle\int\! gu>0$, then there exist two positive numbers $t_u^1$  and $t_u^2$ such that
$t_u^ju\in\Ne_g^V$ and
$t_u^2<T_u<t_u^1<t_u^0$,
where $T_u$ is the unique value for which $T_uu\in\V$.
\item $t^1_u$ and $t^2_u$ depend $C^1$ on $g\in L^\frac{2N}{N+2}$ and on $u\in\D\smallsetminus \{0\}$. Furthermore,
fixed $u$, we have $t^1_u\rightarrow t^0_u$, when $||g||_{L^\frac{2N}{N+2}}\rightarrow0$.
\end{enumerate}
\end{prop}
\begin{proof}
1. If $\varphi_g'(\bar t)=0$, with $\bar t\neq 0$, by \ref{fmu}, we have that
\begin{equation}\label{phi2-g}
{\bar t}^2\varphi_g''(\bar t)=\bar t\int gu+ \int [\bar t u f'(\bar tu)-{\bar t}^2u^2 f''(\bar t u)]<0,
\end{equation}
so $\bar t$ is a maximum point for $\varphi_g$. Furthermore we have that $\varphi_g(0)=0$,
$\varphi_g'(0)>0$ and $\varphi_g''(0)>0$.

Using (\ref{fmu}) and (\ref{f1}), we have
\begin{eqnarray}
\nonumber
\varphi_g(t)&=&\frac{t^2}{2}||u||_V^2- \int_{\R^N} f(tu)-t\int_{\R^N} gu\leq\\
\label{phi-g}
&\leq&\frac{t^2}{2}||u||_V^2-t\int_{\R^N} gu-c_0t^q\int_{t|u|<1}|u|^q
-c_0t^q\int_{t|u|\geq 1}|u|^p\leq\\
\nonumber
&\leq&\frac{t^2}{2}||u||_V^2-t\int_{\R^N}gu-c_0t^p
\int_{|u|\geq 1}|u|^p\rightarrow-\infty\text{ when } t\rightarrow\infty,
\end{eqnarray}
because $p>2$. This proves that there is exactly one $t^1_u$ such that $t^1_u u\in\Ne_g$;
it is easy to see that $t_u^0<t^1_u$.

2. In this case, we can proof, as in (\ref{phi-g}) that $\varphi_g(t)\rightarrow -\infty$ when
$t\rightarrow\infty$ and that if $\bar t\neq 0$ is a critical point of $\varphi_g$ then (\ref{phi2-g})
holds. At last, consider that $0=\varphi_g(0)=\varphi_g'(0)<\varphi_g''(0)$, and so 0 is a
local minimum for $\varphi_g$, and we can conclude.

3. We have just proved that, for any $u\in\D$, we have an unique maximum point $T_u$ of $\varphi_0'(t)$.
So, if we prove that $\int gu<\varphi_0'(T_u)$ we have that there exist two numbers
$t_u^1$ and $t_u^2$ such that $\varphi_g'(t_u^j)=0$.
Set $L$ as in Lemma \ref{L}, and consider that
\begin{equation}
\int gu \leq ||g||_{L^{\frac{2N}{N+2}}}||u||_{L^{2^*}}\leq
C_1||g||_{L^{\frac{2N}{N+2}}}||u||_{\D}\leq C_2 ||g||_{L^{\frac{2N}{N+2}}}||u||_{V}.
\end{equation}
Recalling that $||u||_V=1$, if $||g||_{L^{\frac{2N}{N+2}}}$ is sufficiently small,
that is $C_2 ||g||_{L^{\frac{2N}{N+2}}}<L$,
we have exactly two positive numbers
$t_u^1$ and $t_u^2$ such that $\varphi_g'(t_u^j)=0$, and $t_u^1$ and $t_u^2$ are respectively the
maximum and the minimum point of $\varphi_g$

4. For Simplicity we only prove that $t^1_u(g)$ is a $C^1$ function. The other case is straightforward.
Let us define a function $G:\R^+\times\D\smallsetminus\{0\}\times L^{\frac{2N}{N+2}}\rightarrow \R$,
\begin{equation*}
G:(t,u,g)\mapsto \frac{d}{dt}\varphi_g^u(t)=t||u||^2_V-\int f'(tu)u-\int gu.
\end{equation*}
 We have that $G$ is a $C^1$ function. Let $\bar t,\bar u, \bar g$ be such that
$G(\bar t,\bar u, \bar g)=0.$ We know that
$\frac{\partial}{\partial t}G(\bar t,\bar u, \bar g)=\frac{d^2}{dt^2}\varphi_{\bar g}^{\bar u}(\bar t)<0$,
thus, by the
implicit function theorem there is a $C^1$ function $t(u,g)=t^1_u(g)$ such that $G(t(u,g),u,g)=0$.
We have then the claimed result.
\end{proof}

The Nehari manifold so can be described as:
\begin{equation}
\Ne_g^V=\Ne_{g,V}^+\cup \Ne_{g,V}^-,
\end{equation}
where
\begin{eqnarray*}
\Ne_g^+&=\Ne_{g,V}^+:=&\{ u\in\Ne_g^V\ :\ E_g'(u)u=0, E_g''(u)u^2>0\};\\
\Ne_g^-&=\Ne_{g,V}^-:=&\{ u\in\Ne_g^V\ :\ E_g'(u)u=0, E_g''(u)u^2<0\}.
\end{eqnarray*}
We have also that $E_g^V>0$ on $\Ne_g^-$ and $E_g^V<0$ on $\Ne_g^+$.
Furthermore, because $\Ne_0^V$ and $\V$ are bounded away from 0, we have also that
$\inf\limits_{u\in\Ne_g^-}||u||>0$. The geometry of $\Ne_g^V$ is represented in the following picture.

\begin{center}
\hskip 1.5cm
\includegraphics[scale=.7,angle=0]{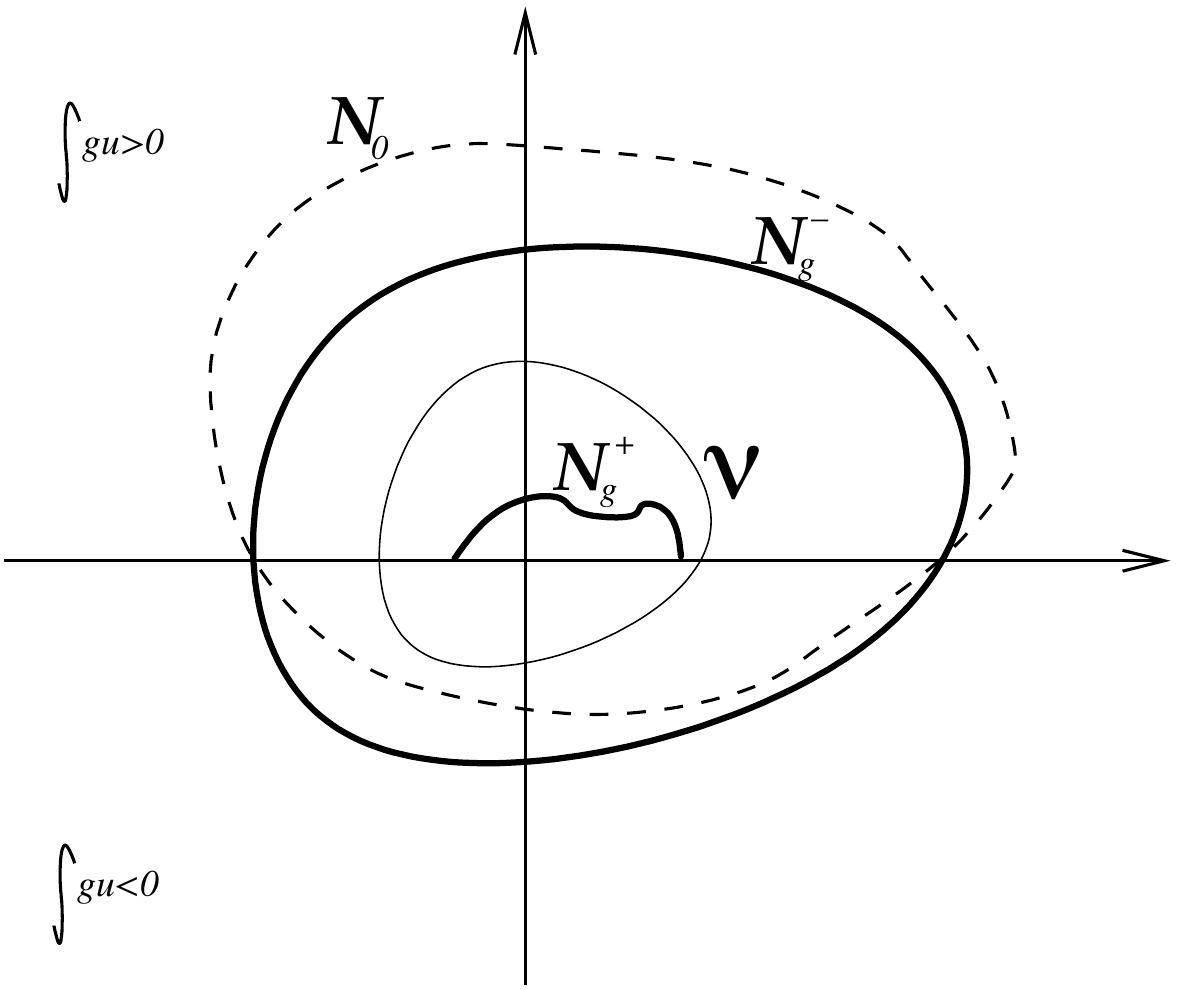}
\end{center}

\begin{rem}\label{u->0}
There exists $M>0$ such that
\begin{equation}
||u||_V\leq M||g||_{\frac{2N}{N+2}}\text { for any }u\in\Ne_g^+,
\end{equation}
indeed, by (\ref{fmu}) we have
\begin{displaymath}
\frac 12 ||u||^2_V<\int f(u)+\int gu\leq \frac 1{\mu_1}\int f'(u)u+\int gu=
\frac 1{\mu_1}||u||^2_V+\left(1- \frac 1{\mu_1}\right)\int gu,
\end{displaymath}
so
\begin{displaymath}
\left(\frac 12- \frac 1{\mu_1}\right)||u||_V^2<\left(1- \frac 1{\mu_1}\right)\int gu.
\end{displaymath}
\end{rem}
\section{The Splitting Lemma}

We recall that a sequence $\{u_n\}_n\in\D$ such that
$E_g^V(u_n)\rightarrow c$,
and
$\nabla E_g^V(u_n)\rightarrow0$
is a Palais-Smale sequence at level $c$ for $E_g^V$.

In the same way we say that $\{u_n\}_n\in\Ne_g^V$ such that
$E_g^V(u_n)\rightarrow c$,
and there exists a sequence $\eps_n\rightarrow0$ s.t.
$|\langle \nabla E_g^V(u_n),\varphi\rangle|\leq \eps_n||\varphi||$, for all
$\varphi\in T_{u_n}\Ne_g^V \cap\D$
is a Palais-Smale sequence at level $c$ for $E_g^V$ restricted to $\Ne_g^V$.

A functional $f$ satisfies the $(PS)_c$ condition if all the
Palais-Smale sequences at level $c$ converge.

Unfortunately the functional $E_g^V$ on $\Ne_g^V$ does not satisfy the {\em PS} condition in
all the energy range.
In this section by the splitting lemma we get a description of the {\em PS} sequences
for the functional $E_g^V$.
\begin{lemma}
Let $u_n\in\Ne_g$ and let $E_g^V(u_n)\rightarrow c$. Then $||u_n||_V$ is bounded.
\end{lemma}
\begin{proof}
We have that
\begin{equation}\label{PS}
||u_n||_V^2=\int f'(u_n)u_n+\int gu_n
\end{equation}
because $u_n\in \Ne_g^V$. Furthermore, for (\ref{fmu}) we have
\begin{eqnarray}
\nonumber
E_g^V(u_n)&=&\frac12 ||u_n||_V^2-\int f(u_n)-\int gu_n\geq \\
\nonumber
&\geq& \frac12 ||u_n||_V^2-\frac{1}{\mu_1}\int f'(u_n)u_n-\int gu_n=\\
\label{Eglimitata}
&=& \frac12||u_n||_V^2-\frac{1}{\mu_1}||u_n||_V^2+\frac{1}{\mu_1}\int gu_n-\int gu_n=\\
\nonumber
&=&\left( \frac12-\frac{1}{\mu_1}\right)||u_n||_V^2-\left(1-\frac{1}{\mu_1}\right)\int gu_n=\\
\nonumber
&=&||u_n||_V^2\left[ \left( \frac12-\frac{1}{\mu_1}\right)-
\left(1-\frac{1}{\mu_1}\right)\int g \frac{u_n}{||u_n||_V^2}\right].
\end{eqnarray}
If $||u_n||_V\rightarrow\infty$ we have that
\begin{equation}
\left|\int g \frac{u_n}{||u_n||_V^2}\right|\leq
||g||_{L^{\frac{2N}{N+2}}}\frac{||u||_{L^{2^*}}}{||u_n||_V^2}\leq
C||g||_{L^{\frac{2N}{N+2}}}\frac{1}{||u_n||_V}\rightarrow0.
\end{equation}
So we will have
\begin{equation}
C_1>E_g^V(u_n)\geq C_2||u_n||_V^2\rightarrow\infty
\end{equation}
that is a contradiction.
\end{proof}
\begin{lemma}\label{splitenergia}
Let $\{u_n\}_n\subset\Ne_g$, and let $E_g^V(u_n)\rightarrow c$. Then, up to subsequence
$u_n\rightharpoonup u_0$ in $\D$. Furthermore, setting $\psi_n=u_n-u_0$ we have
\begin{enumerate}
\item $||\psi_n||_V^2=||u_n||_V^2-||u_0||_V^2+o(1)$;
\item $E_g^V(\psi_n)=E_g^V(u_n)-E_g^V(u_0)+o(1)$.
\end{enumerate}
\end{lemma}
\begin{proof}
By the previous lemma we have that $||u_n||_{\D}$ is bounded. Then
$u_n\rightharpoonup u_0$ and we have that
$$||\psi_n||_V^2=||u_n||_V^2-||u_0||_V^2+o(1).$$
Furthermore, we have that
\begin{equation}
\int f(\psi_n)=\int f(u_n)-\int f(u_0)+o(1).
\end{equation}
Indeed, we have the following equation, where $\tau,\theta,\sigma\in(0,1)$
\begin{eqnarray*}
&&\int f(u_n)-\int f(u_0)-\int f(\psi_n)=\\
&=&\int\limits_{B_R} f(u_0+\psi_n)- f(u_0)-\int\limits_{B_R^C} f(u_0)+\\
&&+\int\limits_{B_R^C} f(u_0+\psi_n)- f(\psi_n)-\int\limits_{B_R} f(\psi_n)=\\
&=&\int\limits_{B_R} f'(u_0+\tau\psi_n)\psi_n-\int\limits_{B_R^C} f(u_0)+
\int\limits_{B_R^C} f'(\theta u_0+\psi_n)u_0-\int\limits_{B_R} f'(\sigma\psi_n)\psi_n.
\end{eqnarray*}
Using Lemma \ref{lp+lq} we have that the terms in $B_R^C$ are arbitrarily small when $R$ is
sufficiently large. Furthermore, since $\psi_n\rightarrow0$ in $L^p(\Omega)$
for all $\Omega\subset \R^N$ bounded and for all
$p<2^*$, we get that
$$ \int f(u_n)-\int f(u_0)-\int f(\psi_n)\rightarrow0.$$
The proof follows easily.
\end{proof}
\begin{lemma}\label{gpsi}
Suppose that $\psi_n\rightharpoonup0$ in $\D$. Then we have
\begin{eqnarray}
\int V\psi_n^2\rightarrow0&&\\
\int g\psi_n\rightarrow0&&
\end{eqnarray}
\end{lemma}
\begin{proof}
Again we use that $\psi_n\rightarrow0$ in $L^p(\Omega)$
for all $\Omega\subset \R^N$ bounded and for all
$p<2^*$. We have that
\begin{eqnarray*}
\int V\psi_n^2&=&\int\limits_{B_R} V\psi_n^2+\int\limits_{\R^N\smallsetminus B_R}V\psi_n^2\leq
||V||_{L^t(B_R)}||\psi_n||^2_{L^{2t'}(B_R)}+\\
&&+
||V||_{L^{N/2}(\R^N\smallsetminus B_R)}||\psi_n||^2_{L^{2^*}(\R^N\smallsetminus B_R)}\rightarrow 0,
\end{eqnarray*}
and that
\begin{eqnarray*}
\int g\psi_n&=&\int\limits_{B_R} g\psi_n +\int\limits_{\R^N\smallsetminus B_R}g\psi_n\leq
||g||_{L^s(B_R)}||\psi_n||_{L^{s'}(B_R)}+\\
&&+
||g||_{L^{\frac{2N}{N+2}}(\R^N\smallsetminus B_R)}
||\psi_n||_{L^{2^*}(\R^N\smallsetminus B_R)}\rightarrow 0.
\end{eqnarray*}
\end{proof}

\begin{lemma}\label{splitting}
Let $\{u_n\}_n$ a PS sequence at level $c$
for the functional $E_g^V$ restricted to the manifold $\Ne_g^V$.
Then, up to a subsequence, there exist $k$ sequences of points $\{y_n^j\}_n$, $j=1,\dots k$,
with $|y_n^j|\rightarrow\infty$,
a solution $u^0$ of the problem $-\Delta u +Vu=f'(u)+g$,
and $k$ solutions $u^j$, $j=1,\dots k$, of the problem
$-\Delta u =f'(u)$ such that
\begin{eqnarray}
u_n(x)&=&u^0(x)+\sum\limits_{j=1}^{k}u^j(x-y_n^j)+o(1);\\
E_g^V(u_n)&=&E_g^V(u^0)+\sum\limits_{j=1}^{k}E_0^0(u^j)+o(1).
\end{eqnarray}
\end{lemma}

\begin{proof}
Since $u_n$ is a {\em PS} sequence for the functional $E_g^V$ restricted to the manifold
$\Ne_g^V$, then $u_n$ is a
{\em PS} sequence for the functional $E_g^V$. By the Lemma \ref{splitenergia} we have that
$u_n$ converges to $u^0$ weakly in $\D$ (up to subsequence), so, given $\varphi\in C^\infty_0(\R^N)$,
\begin{equation}
\lim_{n\rightarrow\infty}\int\nabla u_n\nabla\varphi+Vu_n\varphi-f'(u_n)\varphi-g\varphi=0.
\end{equation}
It is easy to see that
\begin{displaymath}
\int\nabla u_n\nabla\varphi+Vu_n\varphi\rightarrow
\int\nabla u^0\nabla\varphi+Vu^0\varphi.
\end{displaymath}
Arguing as in Step 1 of \cite[Lemma 3.3]{BGM04} we get also that, for some $0<\theta<1$,
\begin{equation}\label{f'(un)}
\int [f'(u_n)-f'(u^0)]\varphi=
\int_{\text{supp} \varphi}f''(\theta u_n+(1-\theta)u^0)(u_n-u^0)\varphi
\rightarrow 0,
\end{equation}
as $n\rightarrow0$, because
$u_n-u^0\rightarrow0$ in $L^p(\Omega)$, with $\Omega$ bounded and $p<2^*$.
So we have proved that
$u^0$ solves $-\Delta u +Vu=f'(u)+g$.

Now we set
\begin{equation*}
\psi_n(x)=u_n(x)-u^0(x).
\end{equation*}
Then $\psi_n\rightharpoonup0$ weakly in $\D$. If $\psi_n\nrightarrow0$ strongly in $\D$,
for Step 3 of \cite[Lemma 3.3]{BGM04} we have that there exists a sequence $\{y_n\}\subset \R^N$
with $|y_n|\rightarrow\infty$ such that $\psi_n(x+y_n)\rightarrow u^1$ in $\D$, and $u^1\neq0$.

Because $u^0$ is a weak solution of (\ref{PV}) and $u_n$ is a $PS$ sequence for $E_g^V$ we have
that, for any $\varphi\in C^\infty_0(\R^N)$,
\begin{eqnarray*}
&&\int\nabla u_n\nabla\varphi+Vu_n\varphi-f'(u_n)\varphi-g\varphi\rightarrow0;\\
&&\int\nabla u^0\nabla\varphi+Vu^0\varphi-f'(u^0)\varphi-g\varphi=0.
\end{eqnarray*}
So
\begin{equation}
\int\nabla \psi_n\nabla\varphi+V\psi_n\varphi-f'(u_n)-f'(u^0)\varphi\rightarrow0.
\end{equation}
Using (\ref{f'(un)}) we have that $\psi_n$ is a $PS$ sequence for the functional $E_0^V$.
Thus, for any $\varphi\in C^\infty_0(\R^N)$ we have

\begin{eqnarray*}
&&\int \nabla\psi_n(x+y_n)\nabla \varphi(x)-f'(\psi_n(x+y_n))\varphi(x)dx=\\
&&\int \nabla\psi_n(x)\nabla \varphi(x-y_n)-f'(\psi_n(x))\varphi(x-y_n)dx=\\
&&\int [f'(u_n)-f'(u^0)-f'(\psi_n)]\varphi(x-y_n)-\int V(x)\psi_n(x)\varphi(x-y_n)+o(1).
\end{eqnarray*}
Using the same argument of Lemma \ref{gpsi} we can prove that
\begin{equation*}
\int V(x)\psi_n(x)v(x-y_n)\leq C\varepsilon_n||\varphi||_{\D},\text{ with }\varepsilon_n\rightarrow0;
\end{equation*}
furthermore we have
\begin{eqnarray*}
&&\int [f'(u_n)-f'(u^0)-f'(\psi_n)]\varphi(x-y_n)=\\
&=&\int\limits_{B_R} [f'(u_0+\psi_n)-f'(u^0)]\varphi(x-y_n)+\\
&&+\int\limits_{B_R^C} [f'(u_0+\psi_n)-f'(\psi_n)]\varphi(x-y_n)-\\
&&-\int\limits_{B_R^C} f'(u^0)\varphi(x-y_n)+\int\limits_{B_R} f'(\psi_n)\varphi(x-y_n)\leq\\
&\leq&||[f''(u^0+\theta\psi_n)-f''(\theta\psi_n)]\varphi(\cdot-y_n)||_{L^{p'}(\R^N)}
||\psi_n||_{L^p(B_R)}+\\
&&+||[f''(\psi_n+\theta u^0)-f''(\theta u^0)]\varphi(\cdot-y_n)||_{L^{p'}\cap L^{q'}}
||u^0||_{L^p+L^q(B_R^C)},
\end{eqnarray*}
where $0<\theta<1$. Because $||u^0||_{L^p+L^q(B_R^C)}\rightarrow0$ for $R\rightarrow\infty$,
and given $R$ $||\psi_n||_{L^p(B_R)}\rightarrow0$ as $n\rightarrow0$, we have that
\begin{equation*}
\int \nabla\psi_n(x+y_n)\nabla \varphi(x)-f'(\psi_n(x+y_n))\varphi(x)dx\rightarrow 0.
\end{equation*}
At last, it is easy to see that
\begin{equation*}
\int \nabla\psi_n(x+y_n)\nabla v(x)-f'(\psi_n(x+y_n))v(x)dx\rightarrow
\int \nabla u^1\nabla v(x)-f'(u^1)v(x)dx,
\end{equation*}
so we have also proved that $u^1$ solves the problem $-\Delta u=f(u)$.

Set $\psi^2_n=\psi(x+y_n)-u^1$, we have that $\psi^2_n\rightharpoonup 0$, thus
\begin{eqnarray*}
E_g^V(u_n)-E_g^V(u^0)&=&E_g^V(\psi_n(x))+o(1)=E_0^0(\psi_n(x))+o(1)=\\
&=&E_0^0(\psi_n(x+y_n))+o(1)=E_0^0(u^1)+E_0^0(\psi^2_n)+o(1),
\end{eqnarray*}
by Lemma \ref{splitenergia}. So,
\begin{equation}
E_g^V(u_n)=E_g^V(u^0)+E_0^0(u^1)+E^0_0(\psi^2_n)+o(1).
\end{equation}
Now, if $\psi^2_n\rightarrow0$ strongly in $\D$, we have the claim,
otherwise we can proceed
by induction and conclude the proof in a finite number of steps.
\end{proof}

\section{Main Results}

We set
\begin{equation*}
m_g=\inf\limits_{u\in \Ne_g^V}E_g^V(u)\text{ and }m_{1,g}=\inf\limits_{u\in \Ne_g^-}E_g^V(u).
\end{equation*}
We show that there exist a solution with critical value $m_g$ and another solution with critical value $m_{1,g}$.

We set also
\begin{equation}
m_0=\inf\limits_{u\in \Ne_0^0}E_0^0(u)
\end{equation}
and we recall that there exists a positive radially symmetric function $\omega\in\Ne_0^0$ such that
\begin{equation}
E_0^0(\omega)=m_0>0.
\end{equation}
Finally, we set
\begin{equation}
m_V =\inf\limits_{u\in \Ne_0^V}E_0^V(u)
\end{equation}
We know, by \cite{BGM04}, that for any $V\leq 0$ and $V<0$ on a set of positive measure  there exists a
function $\bar u \in\Ne_0^V$ such that
\begin{equation}
E_0^V(\bar u)=m_V
\end{equation}
and
\begin{equation}
0<m_V<m_0.
\end{equation}
We prove the following results.
\begin{teo}\label{mainNeg+}
There exist a $u_g\in \Ne_g^+$ such that
$E_g^V(u_g)=m_g$. Furthermore, when $||g||_{L^\frac{2N}{N+2}}$ is small, $u_g$ is unique.

\end{teo}
\begin{proof}
By definition of $\Ne_g^+$ we have that $m_g=\inf\limits_{u\in \Ne_g^+}E_g^V(u)$,
and that $m_g<0$.
At first we prove that $m_g>-\infty$. By contradiction, suppose that there exist a sequence
$t_n>0$ and a sequence $\{v_n\}_n\subset \D$ with $||v_n||_V=1$ and $t_nv_n\in\Ne_g^+$ such that
\begin{equation}
E_g^V(t_nv_n)=\frac{t_n^2}2-\int f(t_nv_n)-t_n\int gv_n\rightarrow-\infty.
\end{equation}
We have also that $\displaystyle t_n^2-\int f'(t_nv_n)t_nv_n-t_n\int g v_n=0$. So,
if $t_n$ is bounded, we have
\begin{eqnarray*}
E_g^V(t_nv_n)&=&-\frac{t_n^2}{2}+\int f'(t_nv_n)t_nv_n-\int f(t_nv_n)\geq\\
&\geq&-\frac{t_n^2}{2}+\left(1-\frac 1{\mu_1}\right)\int f'(t_nv_n)t_nv_n
\end{eqnarray*}
that is bounded by Lemma \ref{lp+lq}. Thus we have that, up to subsequence, $t_n\rightarrow+\infty$.
Finally, arguing as in (\ref{Eglimitata}) we have that
\begin{equation}
E_g^V(t_nv_n)\geq
\left(\frac 12 -\frac 1{\mu_1}\right) t_n^2-\left(1 -\frac 1{\mu_1}\right)t_n\int gv_n\rightarrow+\infty,
\end{equation}
that is a contradiction.

Now, let $u_n$ a minimizing sequence. For the Ekeland variational principle, we can suppose $u_n$ be a
$PS$ sequence. For the splitting lemma there exists a $u_g \in \Ne_g^V$ and $k$ functions
$u^j$,
$1\leq j\leq k$ such that
\begin{equation}
E_g^V(u_n)\rightarrow E_g^V(u_g)+\sum_{j=1}^kE_0^0(u^j)=m_g<0.
\end{equation}
We know that $E_0^0(u^j)\geq m_0>0$ for all $j$. So, if $k>0$ we will have $E_g^V(u_n)\rightarrow m_g+\delta$ for
some $\delta>0$ and this is a contradiction.

So, we have
\begin{equation}
u_n\rightarrow u_g\text{ in }\D.
\end{equation}
Furthermore, we have $E_g^V(u_g)=m_g<0$, so $u_g\in\Ne_g^+$,
and this concludes the proof of the existence.

To prove uniqueness, we argue by contradiction. If $u_1,u_2$ are minimizers of $E^V_g$ on $\Ne_g^+$,
both $u_1$ and $u_2$ solve (\ref{PV}), so we have
\begin{equation*}
||u_1-u_2||_V^2=\int (f'(u_1)-f'(u_2))(u_1-u_2)=\int f''(\theta u_1+(1-\theta)u_2)(u_1-u_2)^2
\end{equation*}
with $0<\theta<1$. So
\begin{equation}
||u_1-u_2||_{L^{2^*}}^2\leq C||u_1-u_2||_V^2\leq C ||u_1-u_2||_{L^{2^*}}^2
||f''(\theta u_1+(1-\theta)u_2)||_{L^\frac{2^*}{2^*-2}}.
\end{equation}
By Remark \ref{u->0}, we have that, if $g\rightarrow0$ in $L^\frac{2N}{N+2}$, then both
$u_1$ and $u_2$ are small in $L^p+L^q$, so we have that $f''(\theta u_1+(1-\theta)u_2)\rightarrow0$ in
$L^{p/p-2}\cap L^{q/q-2}$ by Lemma \ref{stimef}, and, by interpolation,
$$||f''(\theta u_1+(1-\theta)u_2)||_{L^\frac{2^*}{2^*-2}}\rightarrow0,$$
that is a contradiction.
\end{proof}

\begin{prop}\label{upos}
Suppose that $g\geq 0$. Then there exists an $u_g\geq 0$ in $\Ne_g^+$ such that
$E_g^V(u_g)=m_g$.
\end{prop}
\begin{proof}
Take $u_g$ as in Theorem \ref{mainNeg+}.
Because $u_g\in \Ne_g^+$ we have that $\displaystyle \int gu_g>0$.
If $u_g$ changes sign, or $u_g$ negative, we have that
\begin{equation}
0<\int gu_g\leq\int g|u_g|.
\end{equation}
So, reminding that $f$ is even we have
\begin{eqnarray*}
E_g^V(|u_g|)&=&\frac 12||u_g||_V^2 -\int f(|u_g|)-\int g|u_g|\leq\\
&\leq&\frac 12||u_g||_V^2 -\int f(u_g)-\int gu_g=E_g^V(u_g).
\end{eqnarray*}
We know that there exists a $\tau$ such that $\tau |u_g|\in\Ne_g^+$.
Furthermore we know, by the study of $\varphi^{|u_g|}_g$ that $\tau$ is a local minimizer of
$\varphi^{|u_g|}_g$, in fact, $\varphi^{|u_g|}_g(\tau)\leq\varphi^{|u_g|}_g(t)$
for all $t \in [0,\tau]$.
We have
\begin{eqnarray*}
\frac{d}{dt}\varphi^{|u_g|}_g(1)&=&\frac{d}{dt}E_g^V(t|u_g|)_{|_{t=1}}=
||u_g||^2_V-\int f'(|u_g|)|u_g|-\int g|u_g|\leq\\
&\leq&||u_g||^2_V-\int f'(u_g)u_g-\int gu_g=\frac{d}{dt}E_g^V(tu_g)_{|_{t=1}}=0,
\end{eqnarray*}
and
\begin{eqnarray*}
\frac{d^2}{dt^2}\varphi^{|u_g|}_g(1)&=&\frac{d^2}{dt^2}E_g^V(t|u_g|)_{|_{t=1}}=
||u_g||^2_V-\int f''(|u_g|)|u_g|^2=\\
&=&||u_g||^2_V-\int f''(u_g)u_g^2=\frac{d^2}{dt^2}E_g^V(tu_g)_{|_{t=1}}>0.
\end{eqnarray*}
Thus $\tau\geq1$ and
\begin{equation}
E_g^V(\tau|u_g|)\leq E_g^V(|u_g|)\leq E_g^V(u_g)=m_g,
\end{equation}
that concludes the proof.
\end{proof}

We want to prove that, under suitable hypothesis on $g,f$ and $V$,
there exists another solution of \ref{PV}, by minimizing the
functional $E_g^V$ on $\Ne_g^-$.
In order to prove that a minimizing sequence converges
we will show that, for $g$ small,
\begin{equation}
m_{1,g}:=\inf\limits_{u\in\Ne_g^-}E_g^V(u)<m_g+m_0;
\end{equation}
\begin{lemma}\label{limsup}
Suppose that $V\leq0$ and $V<0$ on a set of positive measure.
If $||g||_{L^\frac{2N}{N-2}}$ sufficiently small,
then there exist a $\delta>0$ such that
\begin{equation}
m_{1,g}:=\inf\limits_{u\in\Ne_g^-}E_g^V(u)<m_0-\delta.
\end{equation}
Moreover,
\begin{equation}
\limsup\limits_{||g||_{L^\frac{2N}{N-2}}\rightarrow0}m_{1,g}\leq m_V
\end{equation}
\end{lemma}
\begin{proof}
By \cite[Lemma 4.4(a)]{BGM04} and \cite[Theorem 1.1]{BGM04} we know that
there exists a $\bar u \in \Ne_0^V$ such that
\begin{displaymath}
E_0^V(\bar u)=\inf\limits_{u\in\Ne_0^V}E_0^V(u)=m_V<m_0.
\end{displaymath}

We set $v=\frac{\bar u}{||\bar u||_V}$, so $\bar u=t^v_0v$. We know that there exists $t^v_1=t^v_1(g)$ such that
$t^v_1v\in \Ne^-_g$ by Proposition \ref{gsmall}. Furthermore,
by Proposition \ref{gsmall} we have that $t^v_1\rightarrow t^v_0$
when $||g||_{L^\frac{2N}{N-2}}\rightarrow0$, and so
\begin{equation}
m_{1,g}\leq E^V_g(t^v_1 v)\rightarrow E^V_0(\bar u)=m_V<m_0\,\text{ for }||g||_{L^\frac{2N}{N+2}}\rightarrow0,
\end{equation}
that concludes the proof.
\end{proof}
\begin{teo}\label{u1g}
For $||g||_{L^\frac{2N}{N-2}}\rightarrow0$ there exist $u_{1,g}\in \Ne_g^-$ a solution of (\ref{PV}). Furthermore,
if $g\geq 0$ the solution $u_{1,g}$ can be chosen positive.
\end{teo}
\begin{proof}
By the splitting lemma, to obtain the result it is enough to show that
$m_{1,g}<m_g+m_0$.
In the previous lemma, we have proved that there exists a $\delta>0$ such that
$m_{1,g}<m_0-\delta$ for $||g||_{L^\frac{2N}{N-2}}$ sufficiently small.
By Remark \ref{gsmall} we have also that $m_g\rightarrow0$ when $g\rightarrow 0 $ in $L^\frac{2N}{N+2}$.
So there exists $u_{1,g}\in \Ne_g^V$ a solution of (\ref{PV}). Moreover $E_g^V(u_{1,g})$ is positive, so
$u_{1,g}\in\Ne_V^-$.

To prove the last claim, consider that
$E_g^V(|u_{1,g}|)\leq E_g^V(u_{1,g})$. Also, there exists a $\bar t$ such that $\bar t |u_{1,g}|\in \Ne_g^-$.
Then we have
\begin{equation}
m_{1,g}=E_g^V(u_{1,g})=\max_{t}E_g^V(t u_{1,g})\geq E_g^V(\bar t u_{1,g})\geq E_g^V(\bar t |u_{1,g}|).
\end{equation}
So if $u_{1,g}$ is a solution, also $\bar t |u_{1,g}|\in \Ne_g^-$ is a solution of (\ref{PV}).
\end{proof}
\begin{prop}\label{mv}
If $||g||_{L^{p'}\cap L^{q'}}\rightarrow0$, then $m_{1,g}\rightarrow m_V$.
\end{prop}
\begin{proof}
We take a sequence of $g_n\rightarrow 0$ in $L^{p'}\cap L^{q'}$. We know that for any $g_n$ there exists
$u_{1,{g_n}}$ such that $E_{g_n}^V(u_{1,{g_n}})=m_{1,{g_n}}$. For simplicity we call $u_n=u_{1,{g_n}}$.
Also, we set $v_n=\frac{u_n}{||u_n||_{L^{p}+ L^{q}}}$, and $u_n=t_nv_n$.
We have
\begin{equation}
E_{g_n}^V(u_n)=t_n\left[
\frac 12 \int f'(t_nv_n)v_n-\int \frac{f(t_n v_n)}{t_n}-\frac 12 \int g_nv_n,
\right]
\end{equation}
and we have that there exist a $\delta>0$ such that $0\leq E_{g_n}^V(u_n)\leq m_v+\delta$.
Now, suppose, by contradiction, that $t_n\rightarrow\infty$.
Then,
\begin{equation}
\frac 12 \int f'(t_nv_n)v_n-\int \frac{f(t_n v_n)}{t_n}-\frac 12 \int g_nv_n\rightarrow 0,
\end{equation}
and so
\begin{equation}\label{ftnvn}
\frac 12 \int f'(t_nv_n)v_n-\int \frac{f(t_n v_n)}{t_n}\rightarrow 0.
\end{equation}
By (\ref{fmu}), we have that
\begin{eqnarray*}
\int f'(t_nv_n)v_n-2\int \frac{f(t_n v_n)}{t_n}&=&
\int f'(t_nv_n)v_n-\mu_1\int \frac{f(t_n v_n)}{t_n}+\\
&&+(\mu_1-2)\int \frac{f(t_n v_n)}{t_n}\geq\\
&\geq&(\mu_1-2)\int \frac{f(t_n v_n)}{t_n}.
\end{eqnarray*}
So $\displaystyle\int \frac{f(t_n v_n)}{t_n}\rightarrow0$.
Now the hypothesis on $f$
\begin{equation}
0\leq c_0 t_n^{p-1}\left[\int\limits_{|v_n>1|}|v_n|^p+\int\limits_{|v_n<1|}|v_n|^q \right]\leq
\int \frac{f(t_n v_n)}{t_n}\rightarrow0,
\end{equation}
so we have that both
$\displaystyle \int \limits_{|v_n|>1}|v_n|^p$ and
$\displaystyle \int \limits_{|v_n|<1}|v_n|^q$ vanish when $n\rightarrow \infty$, and so
\begin{equation}
1=||v_n||_{L^p+L^q}\leq\max\left\{\int \limits_{|v_n|>1}|v_n|^p,\int \limits_{|v_n|<1}|v_n|^q\right\}\rightarrow0
\end{equation}
that is a contradiction.
Furthermore, by Proposition \ref{gsmall}, we have $t_n$ bounded away from 0.
So, we have that there exists two positive constants $c_1$ and $c_2$ such that
\begin{equation}
0<c_1\leq t_n=||u_n||_{L^p+L^q}\leq c_2<\infty.
\end{equation}

Now, let $\tau_n$ such that $\tau_n u_n\in \Ne_0^V$. We can show that $\tau_n\rightarrow1$ when
$n\rightarrow\infty$.
The main idea is that
\begin{equation}
\frac{d}{dt}\varphi_{g_n}^{u_n}(\tau_n)-\frac{d}{dt}\varphi_{0}^{u_n}(\tau_n)=\int g_nu_n\rightarrow0
\end{equation}
because $||u_n||_{L^p+L^q}$ is bounded and $g_n\rightarrow 0$ in $L^{p'}\cap L^{q'}$.
The details are omitted for the sake of simplicity.

Now we have that
\begin{equation}\label{taunun}
E_0^V(\tau_nu_n)-E_0^V(u_n)\rightarrow 0.
\end{equation}
We have that $E_{g_n}^V(u_n)$ is bounded, so, up to subsequences, there exists a $d$ such that
$E_{g_n}^V(u_n)\rightarrow d$ when $n\rightarrow\infty$, and, because $u_n$ is bounded in $L^p+L^q$, also
$E_0^V(u_n)\rightarrow d$, and, by (\ref{taunun}), $E_0^V(\tau_n u_n)\rightarrow d$.

So, $d\geq m_V$. By Lemma \ref{limsup} we know also that $d\leq m_V$ so we get the claim.
\end{proof}

\begin{proof}[Proof of Theorem 1]
By theorems \ref{mainNeg+} and \ref{u1g}, we have that there exists a $u_g\in \Ne_g^+$ and $u_{1,g}\in \Ne_g^-$
that solve (\ref{PV}). Furthermore, by Theorem \ref{u1g} and Proposition \ref{upos} the solution can be chosen
nonnegative. At least, by Remark \ref{u->0} we have that $u_g\rightarrow0$ in $\D$ and by Proposition \ref{mv}
that $m_{1,g}\rightarrow m_V$ when $g\rightarrow0$.
\end{proof}

\appendix

\section{The Hypothesis on $f$}

We want to prove that there exists a function that satisfies all the conditions
required in the introduction.

We take the function
\begin{equation}
f(s)=\frac{|s|^{q}}{1+|s|^{q-p}}.
\end{equation}
This function is even, and it satisfies (\ref{f1}).

We have that, for $s>0$
\begin{equation*}
f^{\prime }(s)=\frac{qs^{q-1}+ps^{2q-p-1}}{(1+s^{q-p})^{2}},
\end{equation*}

\begin{displaymath}
f^{\prime \prime }(s)=\frac{s^{q-2}}{(1+s^{q-p})^{2}}\left\{
q(q-1)+p(2q-p-1)s^{q-p}-\frac{2(q-p)(q+ps^{q-p})s^{q-p}}{1+s^{q-p}}\right\}.
\end{displaymath}
It's easy to see that $f$ satisfies (\ref{f2}) and the first part of (\ref{fmu}).

We set $\mu _{2}=1+\varepsilon >1$; then the inequality
$(1+\varepsilon )f^{\prime }(s)s<f^{\prime \prime }(s)s^{2}$ becomes
\begin{displaymath}
(q^{2}-2q-\varepsilon q)+p(2q-p-2-\varepsilon )\gamma -
\frac{2(q-p)(q+p\gamma )\gamma }{1+\gamma }>0,
\end{displaymath}
where $\gamma =s^{q-p}$. So, we have to prove that

\begin{displaymath}
q(q-2-\varepsilon )+[p(2q-p-2-\varepsilon )+q(2p-q-2-\varepsilon )]\gamma
+p(p-2+\varepsilon )\gamma ^{2}>0.
\end{displaymath}

Obviously we can choose $\varepsilon $ such that $q(q-2-\varepsilon )>0$ and
$p(p-2+\varepsilon )>0$. Furthermore, we choose $q-p$ sufficiently small
such that also $2q-p-2-\varepsilon$ and $2p-q-2-\varepsilon$ are positive, so
the second part of
(\ref{fmu}) is proved.

At last we prove (\ref{f3}) and that $f'''(s)s^3>0$. We have that, for $s>0$,
\begin{eqnarray*}
f^{\prime \prime \prime }(s) &=&\frac{6(p-q)^{3}s^{4q-3p-3}}{(1+s^{q-p})^{4}}
-\frac{6(1+p-2q)(p-q)^{2}s^{3q-2p-3}}{(1+s^{q-p})^{3}}+ \\
&+&\frac{(2p+3p^{2}+p^{3}-2q-12pq-6p^{2}q+9q^{2}+12pq^{2}-7q^{3})s^{2q-p-3}}
{(1+s^{q-p})^{2}}+ \\
&+&\frac{q(2-3q+q^{2})s^{q-3}}{1+s^{q-p}}.
\end{eqnarray*}
We obtain that
\begin{displaymath}
f^{\prime \prime \prime }(s)s^{3}=\frac{As^{q}}{1+s^{q-p}}+
\frac{Bs^{2q-p}}{(1+s^{q-p})^{2}}+\frac{Cs^{3q-2p}}{(1+s^{q-p})^{3}}+
\frac{Ds^{4q-3p}}{(1+s^{q-p})^{4}},
\end{displaymath}
were
\begin{equation*}
\begin{array}{ll}
A =q(q-2)(q-1);&
B =(p-q)(2+3p+p^{2}-9q-5pq+7q^{2}); \\
C =6(p-q)^{2}(2q-p-1);&
D =6(p-q)^{3}.
\end{array}
\end{equation*}

We can choose $q-p$ sufficiently small,
in order to have $B,C,D<<A$. Now, set as above $\gamma =s^{q-p}$, we have
\begin{displaymath}
f'''(s)s^3=\frac{s^{q}
\left[
A+(3A+B)\gamma +(3A+2B+C)\gamma ^{2}+(A+B+C+D)\gamma ^{3}
\right]
}{(1+s^{q-p})^{4}}
\end{displaymath}
that is positive for all $s>0$. So (\ref{fmu}) is completely proved.

Furthermore, we have that
\begin{equation}
\lim\limits_{s\rightarrow0^+}\frac{f'''(s)}{s^{q-3}}=A=q(q-1)(q-2)>0,
\end{equation}
and
\begin{equation}
\lim\limits_{s\rightarrow+\infty}\frac{f'''(s)}{s^{p-3}}=A+B+C+D=p(p-1)(p-2)>0.
\end{equation}
So, there exists a $c_3>0$ such that
\begin{equation}
\left\{
\begin{array}{ll}
|f'''(s)|\leq c_{3}|s|^{p-3} & \text{ for }|s|\geq 1; \\
|f'''(s)|\leq c_{3}|s|^{q-3} & \text{ for }|s|\leq 1.
\end{array}
\right.
\end{equation}
Now, let $\Gamma=\{x\in\R^N\ :\ |u(x)|>1\}$ and $\Delta=\R^N\smallsetminus  \Gamma$
We have that
\begin{eqnarray*}
\int f'''(u)u^3&\leq &\int_\Gamma f'''(u)u^3+\int_\Delta f'''(u)u^3
\leq c_3\int_\Gamma|u|^p+c_3\int_\Delta|u|^q\leq\\
&\leq&C_1+C_2||u||_{L^p+L^q}\leq C_3+C_4||u||_{\D}<\infty,
\end{eqnarray*}
and this proves (\ref{f3}).

\nocite{Ne60}
\nocite{St84}
\nocite{BL83a}
\nocite{BL83b}
\nocite{BR05}
\nocite{AP}
\nocite{BN89}


\begin{thebibliography}{10}

\bibitem{AP}
Antonio Azzollini and Alessio Pomponio, \emph{Compactness result and
  applications to some ``zero mass'' elliptic problems}, ArXiv preprint.

\bibitem{BR05}
Marino Badiale and Sergio Rolando, \emph{Elliptic problems with singular
  potential and double-power nonlinearity}, Mediterr. J. Math. \textbf{2}
  (2005), no.~4, 417--436.

\bibitem{BN01}
Soohyun Bae and Wei-Ming Ni, \emph{Existence and infinite multiplicity for an
  inhomogeneous semilinear elliptic equation on {$\bold R\sp n$}}, Math. Ann.
  \textbf{320} (2001), no.~1, 191--210.

\bibitem{BF04}
Vieri Benci and Donato Fortunato, \emph{Towards a unified field theory for
  classical electrodynamics}, Arch. Ration. Mech. Anal. \textbf{173} (2004),
  no.~3, 379--414.

\bibitem{BGM04}
Vieri Benci, Carlo~R. Grisanti, and Anna~Maria Micheletti, \emph{Existence and
  non existence of the ground state solution for the nonlinear schroedinger
  equations with {$V(\infty)=0$}}, Topol. Methods Nonlinear Anal. \textbf{26}
  (2005), 203--219.

\bibitem{BM04}
Vieri Benci and Anna~Maria Micheletti, \emph{Solutions in exterior domains of
  null mass nonlinear field equations}, Advanced nonlinear studies \textbf{6}
  (2006), no.~2, 171--198.

\bibitem{BL83a}
Henry Berestycki and Pierre-Louis Lions, \emph{Nonlinear scalar field
  equations. {I}. {E}xistence of a ground state}, Arch. Rational Mech. Anal.
  \textbf{82} (1983), no.~4, 313--345.

\bibitem{BL83b}
Henry Berestycki and Pierre-Louis Lions, \emph{Nonlinear scalar field equations. {II}. {E}xistence of
  infinitely many solutions}, Arch. Rational Mech. Anal. \textbf{82} (1983),
  no.~4, 347--375.

\bibitem{BL76}
J{\"o}ran Bergh and J{\"o}rgen L{\"o}fstr{\"o}m, \emph{Interpolation spaces.
  {A}n introduction}, Springer-Verlag, Berlin, 1976, Grundlehren der
  Mathematischen Wissenschaften, No. 223.

\bibitem{Be96}
Guy Bernard, \emph{An inhomogeneous semilinear equation in entire space}, J.
  Differential Equations \textbf{125} (1996), no.~1, 184--214.

\bibitem{BN89}
Haim Brezis and Louis Nirenberg, \emph{A minimization problem with critical
  exponent and non zero data}, Symmetry in Nature, Scuola Normale Superiore,
  Pisa, 1989, pp.~129--140.

\bibitem{GM06}
Marco Ghimenti and Anna~Maria Micheletti, \emph{Existence of minimal nodal
  solutions for the nonlinear {S}chr\"odinger equations with {$V(\infty)=0$}},
  Adv. Differential Equations \textbf{11} (2006), no.~12, 1375--1396.

\bibitem{Ne60}
Zeev Nehari, \emph{On a class of nonlinear second-order differential
  equations}, Trans. Amer. Math. Soc. \textbf{95} (1960), 101--123.

\bibitem{Pis}
Lorenzo Pisani, \emph{Remark on the sum of lebesgue spaces}, Preprint of the
  University of Bari.

\bibitem{St84}
Michael Struwe, \emph{A global compactness result for elliptic boundary value
  problems involving limiting nonlinearities}, Math. Z. \textbf{187} (1984),
  no.~4, 511--517.

\bibitem{Ta92}
G.~Tarantello, \emph{On nonhomogeneous elliptic equations involving critical
  {S}obolev exponent}, Ann. Inst. H. Poincar\'e Anal. Non Lin\'eaire \textbf{9}
  (1992), no.~3, 281--304.

\bibitem{Zho00}
Huan-Song Zhou, \emph{Solutions for a quasilinear elliptic equation with
  critical {S}obolev exponent and perturbations on {${\bf R}\sp N$}},
  Differential Integral Equations \textbf{13} (2000), no.~4-6, 595--612.

\bibitem{Zhu91}
Xi~Ping Zhu, \emph{A perturbation result on positive entire solutions of a
  semilinear elliptic equation}, J. Differential Equations \textbf{92} (1991),
  no.~2, 163--178.

\end{thebibliography}
\end{document}